\documentclass[12pt]{article}
\usepackage{amssymb,latexsym,amsmath,amsthm,amsfonts, enumerate}
\usepackage{color}
\usepackage[all]{xy}

\usepackage{hyperref}
\usepackage{tikz}
\usetikzlibrary{shapes,snakes}

\def\Hom{\mathop{\rm Hom}\nolimits}
\def\Ext{\mathop{\rm Ext}\nolimits}

\def\mod{\mathop{\rm mod}\nolimits}

\def\add{\mathop{\rm add}\nolimits}

\def\tilt{\mathop{\rm tilt}\nolimits}

\def\silt{\mathop{\rm silt}\nolimits}

\def\tilt{\mathop{\rm \tau\makebox{-}tilt}\nolimits}
\def\stilt{\mathop{\rm s\tau\makebox{-}tilt}\nolimits}

\def\H{\mathop{\rm H}\nolimits}
\def\rad{\mathop{\rm rad}\nolimits}
\def\t{\mathop{\rm tilt}\nolimits}
\def\Soc{\mathop{\rm soc}\nolimits}

\usepackage{ulem}
\begin{document}

\newcommand{\nc}{\newcommand}



\newtheorem{theorem}{Theorem}[section]
\newtheorem{proposition}[theorem]{Proposition}
\newtheorem{lemma}[theorem]{Lemma}
\newtheorem{corollary}[theorem]{Corollary}
\newtheorem{conjecture}[theorem]{Conjecture}
\newtheorem{question}[theorem]{Question}
\newtheorem{definition}[theorem]{Definition}
\newtheorem{example}[theorem]{Example}

\newtheorem{remark}[theorem]{Remark}
\def\Pf#1{{\noindent\bf Proof}.\setcounter{equation}{0}}
\def\>#1{{ $\Rightarrow$ }\setcounter{equation}{0}}
\def\<>#1{{ $\Leftrightarrow$ }\setcounter{equation}{0}}
\def\bskip#1{{ \vskip 20pt }\setcounter{equation}{0}}
\def\sskip#1{{ \vskip 5pt }\setcounter{equation}{0}}
\def\bg#1{\begin{#1}\setcounter{equation}{0}}
\def\ed#1{\end{#1}\setcounter{equation}{0}}
\def\KET{T^{^F\bot}\setcounter{equation}{0}}
\def\KEC{C^{\bot}\setcounter{equation}{0}}

\renewcommand{\thefootnote}{\fnsymbol{footnote}}
\setcounter{footnote}{0}
%
%


\title{\bf $\tau$-Tilting modules over one-point extensions by a simple module at a source point
\thanks{This work was partially supported by NSFC (Grant No. 11971225). } }
\footnotetext{
E-mail:~hpgao07@163.com}
\smallskip
\author{\small Hanpeng Gao\\
{\it \footnotesize Department of Mathematics, Nanjing University, Nanjing 210093, Jiangsu Province, P.R. China}\\
}

\date{}
\maketitle
\baselineskip 15pt
%
%
\begin{abstract}
Let $B$ be an one-point extension of a finite dimensional $k$-algebra $A$  by a simple $A$-module at a source point $i$. In this paper, we  classify the $\tau$-tilting modules over $B$. Moreover, it is
shown that there are equations
 $$|\tilt B|=|\tilt A|+|\tilt A/\langle e_i\rangle|\quad \text{and}\quad |\stilt B|=2|\stilt A|+|\stilt A/\langle e_i\rangle|.$$
 As a consequence, we can calculate  the numbers of  $\tau$-tilting modules and support $\tau$-tilting modules over linearly Dynkin type algebras whose
square radical are zero. 
\vspace{10pt}

\noindent {\it 2020 Mathematics Subject Classification}: 16G20, 16G60.


\noindent {\it Key words}: support $\tau$-tilting modules, one-point extensions, Dynkin type algebras.

\end{abstract}
%
\vskip 30pt

\section{Introduction}
As a generalization of tilting module,  the concept of support $\tau$-tilting modules is introduced by  Adachi, Iyama and Reiten\cite{AIR}. They are very important in representation theory of algebras because they  are in bijection with some important objects  including
functorially finite torsion classes,
 2-term silting complexes,  cluster-tilting objects.  It is  very interesting  to calculate  the number of support $\tau$-tilting modules over a given algebra.

 For Dynkin type algebras $\Delta_n$, the numbers of tilting modules and support tilting modules  were first calculated  in \cite{FZYsyst} via cluster algebras, and later in \cite{ONFR2015} via representation theory.

 Recall that a finite-dimensional $k$-algebra is said to be a $Nakayama~ algebra$  if every indecomposable projective module and every indecomposable injective module has a unique composition series.  May authors calculate  the numbers of  $\tau$-tilting modules and support $\tau$-tilting modules over Nakayama algebras. In particular, for square radical zero Nakayama algebra $\Lambda^2_n$ with $n$ simple modules,  there are  the following  recurrence relations (see, \cite{Adachi2016,Asai2018,GS2020}),
 $$|\tilt \Lambda^2_n|=|\tilt \Lambda^2_{n-1}|+|\tilt \Lambda^2_{n-2}|\quad \text{and}\quad |\stilt \Lambda^2_n|=2|\stilt \Lambda^2_{n-}|+|\stilt \Lambda^2_{n-2}|.$$

In this paper, we consider $\tau$-tilting modules and support $\tau$-tilting modules over the  one-point extension $B$ of $A$  by a simple $A$-module at a source point $i$. We will show that there is a bijection
$$\tilt B\mapsto \tilt A\coprod \tilt A/\langle e_i\rangle.$$
We also  get the following  equations,
 $$|\tilt B|=|\tilt A|+|\tilt A/\langle e_i\rangle|\quad \text{and}\quad |\stilt B|=2|\stilt A|+|\stilt A/\langle e_i\rangle|.$$
 As an application , we can calculate  the numbers of  $\tau$-tilting modules and support $\tau$-tilting modules over linearly Dynkin type algebras whose
square radical are zero.

Throughout this paper, all algebras will be basic,  connected, finite dimensional $k$-algebras over an algebraically closed field $k$.  For an algebra $A$, we denote by $\mod A$
 the category of finitely generated left $A$-modules and by $\tau$ the Auslander-Reiten
  translation of $A$.    Let  $P_i$ be the indecomposable projective module and $S_i$ the simple module of $A$ corresponding to the point $i$ for $i=1,2,\cdots,n$.   For  $M\in \mod A$, we also  denote by $|M|$ the number of pairwise nonisomorphic indecomposable summands of $M$  and by $\add M$  the full subcategory of $\mod A$ consisting of direct summands of finite direct sums of copies of $M$. For a set $X$, we denote by $|X|$ the cardinality of $X$.  For two sets $X,Y$, $X\coprod Y$ means the disjoint union.

\section{Main results}\label{sect 2}

Let $A$ be an algebra.  we   recall the definition about support $\tau$-tilting modules.

\begin{definition}\label{2.1} {\rm (\cite[Definition 0.1]{AIR})}
Let $M\in\mod A$.
\begin{enumerate}
\item[(1)] $M$ is  {\it $\tau$-rigid} if $\Hom_A(M,\tau M)=0$.
\item[(2)] $M$ is  {\it $\tau$-tilting}  if it is $\tau$-rigid and $|M|=|A|$.
\item[(3)] $M$ is {\it support $\tau$-tilting} if it is a $\tau$-tilting $A/AeA$-module
for some idempotent $e$ of $A$.
\end{enumerate}
\end{definition}

We will denote by $\tilt A$ (respectively, $\stilt A$) the set of isomorphism classes of basic $\tau$-tilting (respectively, support $\tau$-tilting) $A$-modules.

Let  $X\in \mod A$. The {\it one-point extension} of $A$ by $X$ is defined as the following matrix algebra
\begin{center}
$B=\left(\begin{matrix}
A &X\\
0 & k\\
\end{matrix}\right)$
\end{center}
with the ordinary matrix addition and the multiplication, and we write $B:=A[X]$ with $a$ the extension point.

Let $A$ be an algebra with  a source point $i$, in this paper, we always assume $B:=A[S_i]$. In this case, $P_a$ is an indecomposable projective-injective $B$-module and $S_a$ is  a simple injective $B$-module  by \cite[Proposition 2.5(c)]{ARS}.

\begin{lemma}\label{2.2} Let $M\in \mod B$. If $M$ is $\tau$-rigid, then $M\oplus P_a$ is also. Moreover,
if $M$ is $\tau$-tilting, then it have $P_a$ as a direct summand.
\end{lemma}

\begin{proof}
Since $S_i$ is simple,  there are only two indecomposable $B$-modules $P_a$, $S_a$ which have $S_a$ as a composition factor and they are injective, we have $\tau M$ have no  $S_1$ as
a composition factor.Thus, $\Hom_A(P_a, \tau M)=0$ and, we get  $M\oplus P_a$ is $\tau$-rigid. If $M$ is $\tau$-tilting, then it is maximal $\tau$-rigid by \cite[Theorem 2.12]{AIR}. Hence,  it have $P_a$ as a direct summand.
\end{proof}

\begin{theorem}\label{2.3}
There is a  bijection
$$\tilt B\longleftrightarrow \tilt A\coprod \tilt A/\langle e_i\rangle.$$
\end{theorem}

\begin{proof}
Let $M\in  \tilt A\coprod \tilt A/\langle e_i\rangle$. If $M\in \tilt A$, then $\tau M$ has no $S_a$ as a composition factor since the vertex $a$ is a source in $B$, and hence $\Hom_{B}(P_a,\tau M)=0$. Therefore,  $M\oplus P_a$ is a $\tau$-tilting $B$-module since it is $\tau$-rigid and $|M\oplus P_a|=|M|+1=|A|+1=|B|$. If $M\in \tilt A/\langle e_i\rangle$, then $M$ has no $S_i$ as a composition factor  and $\tau M$ has no $S_a$ as a composition factor.
Note that there is an almost split sequence $0\to S_i\to P_a\to S_a\to 0$, we have $\tau S_a=S_i$. Thus,
$$\Hom_{B}(M\oplus P_a\oplus S_a, \tau(M\oplus P_1\oplus S_a))=\Hom_{B}(M\oplus P_a\oplus S_a, \tau M\oplus S_i)=0.$$
So,  $M\oplus P_a\oplus S_a$ is a $\tau$-tilting $B$-module.

Conversely, Let $M\in \tilt B$. Then we decompose $M$ as $M=P_a\oplus N$ by Lemma \ref{2.2}. If $N$ has no $S_a$ as direct summand, the $N$ is a $\tau$-tilting $B/\langle e_a\rangle(\cong A)$-module, that is,
$N\in \tilt A$. If $N$ has $S_a$ as direct summand, then  we decompose $N$ as $N=S_a\oplus L$ where $L$ has no $S_a$ as a composition factor. We claim that
$L$ has no $S_i$ as a composition factor.  Otherwise,  there is a summand $K$ of $L$ such that the top of $K$ is $S_i$ since $i$ is a source point. In particular, $\Hom_{B}(L, S_i)\neq 0$. This implies
$$\Hom_{B}(M, \tau M)=\Hom_{B}(L\oplus P_a\oplus S_a, \tau L\oplus S_i)\neq 0.$$  This is a contradiction.  Hence, $L$ is a $\tau$-tilting $A/\langle e_i\rangle$-module , that is,  $L\in  \tilt A/\langle e_i\rangle$.

\end{proof}

\begin{corollary}\label{a1} All $\tau$-tilting $B$-modules are exactly those forms $P_a\oplus M_1$ and $P_a\oplus S_a\oplus M_2$ where $M_1$ and $M_2$ are $\tau$-tilting modules over $A$  and  $A/\langle e_i\rangle$ respectively.

\end{corollary}

The above Corollary give a relation about $|\tilt B|$ and $|\tilt A|$.

\begin{corollary} \label{2.4} We have
 $$|\tilt B|=|\tilt A|+|\tilt A/\langle e_i\rangle|.$$
\end{corollary}

Let $A$ be an algebra and $M\in\mod A$. $M$ is called  a (classical) $tilting$ module if
\begin{enumerate}
\item[(1)] The projective dimension of $M$ is at most one.
\item[(2)] $\Ext_A^1(M,M)=0$.
\item[(3)] $|M|=|A|$.
\end{enumerate}

Hence, an $A$-module $M$ is tilting if and only if it is a $\tau$-tilting and
its projective dimension is at most one by  the Auslander-Reiten formula. The set of all tilting $A$-modules will be denoted by $\t  A$.

\begin{corollary}\label{a2} Let $A$ be an algebra with a source $i$. Assume that $i$ is not a sink and $B:=A[S_i]$.   All tilting $B$-modules are exactly those forms $P_a\oplus M_1$  where $M_1$ is a  tilting module over $A$. In particular, $|\t B|=|\t A|$.\end{corollary}
\begin{proof} By Corollary \ref{a1},  All $\tau$-tilting $B$-modules are exactly those forms $P_a\oplus M_1$ and $P_a\oplus S_a\oplus M_2$ where $M_1$ and $M_2$ are $\tau$-tilting modules over $A$  and  $A/\langle e_i\rangle$ respectively.
Note that the projective dimension of $M_1$ as $A$-module is equal to the projective dimension of $M_1$ as $B$-module since $a$ is a source of $B$. Hence $P_a\oplus M_1$ is a tilting  $B$-module if and only if $M_1$ is a tilting $A$-module.

Since there is an exact sequence   $0\to S_i\to P_a\to S_a\to 0$ in $\mod B$, we have the projective dimension of $S_a$ is at least two since $S_i$ is not projective when  $i$ is not a sink. Hence  $P_a\oplus S_a\oplus M_2$ is not tilting.
Thus, $|\t B|=|\t A|$.
\end{proof}

\begin{example}\label{}{\rm

Let $B$ be a algebra given by the quiver $$
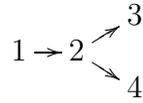
\xymatrix@R=0.00001PT@C=10PT{&&3\\
1\ar[r]&2\ar[ru]\ar[rd]&\\
&&4}$$
with $\rad^2=0$.
Assume that $A$ is the path algebra given by the quiver $\xymatrix@C=10PT{3&2\ar[l]\ar[r]&4}$, we have $B=A[2]$.
There are  five $\tau$-tilting $A$-modules  as follows (there are exactly all tilting-$A$-modules since $A$ is hereditary)
$${\smallmatrix 3
\endsmallmatrix}{\smallmatrix 2\\3~4
\endsmallmatrix}{\smallmatrix 4
\endsmallmatrix},~\quad{\smallmatrix 2\\4
\endsmallmatrix}{\smallmatrix 2\\3~4
\endsmallmatrix}{\smallmatrix 4
\endsmallmatrix},~\quad {\smallmatrix 3
\endsmallmatrix}{\smallmatrix 2\\3~4
\endsmallmatrix}{\smallmatrix 2\\3
\endsmallmatrix},~\quad {\smallmatrix 2\\4
\endsmallmatrix}{\smallmatrix 2\\3~4
\endsmallmatrix}{\smallmatrix 2
\endsmallmatrix},~\quad {\smallmatrix 2
\endsmallmatrix}{\smallmatrix 2\\3~4
\endsmallmatrix}{\smallmatrix 2\\3
\endsmallmatrix}.~$$
We only have one $\tau$-tilting $A/\langle e_2\rangle$-module ${\smallmatrix 3
\endsmallmatrix}{\smallmatrix 4
\endsmallmatrix}$. Hence, we get all $\tau$-tilting $B$-modules by Corollary \ref{a1}.
$${\smallmatrix 1\\2
\endsmallmatrix}{\smallmatrix 3
\endsmallmatrix}{\smallmatrix 2\\3~4
\endsmallmatrix}{\smallmatrix 4
\endsmallmatrix},~\quad{\smallmatrix 1\\2
\endsmallmatrix}{\smallmatrix 2\\4
\endsmallmatrix}{\smallmatrix 2\\3~4
\endsmallmatrix}{\smallmatrix 4
\endsmallmatrix},~\quad {\smallmatrix 1\\2
\endsmallmatrix}{\smallmatrix 3
\endsmallmatrix}{\smallmatrix 2\\3~4
\endsmallmatrix}{\smallmatrix 2\\3
\endsmallmatrix},~\quad {\smallmatrix 1\\2
\endsmallmatrix}{\smallmatrix 2\\4
\endsmallmatrix}{\smallmatrix 2\\3~4
\endsmallmatrix}{\smallmatrix 2
\endsmallmatrix},~\quad {\smallmatrix 1\\2
\endsmallmatrix}{\smallmatrix 2
\endsmallmatrix}{\smallmatrix 2\\3~4
\endsmallmatrix}{\smallmatrix 2\\3
\endsmallmatrix},~\quad {\smallmatrix 1\\2
\endsmallmatrix}{\smallmatrix 1
\endsmallmatrix}{\smallmatrix 3
\endsmallmatrix}{\smallmatrix 4
\endsmallmatrix}.$$
By Corollary \ref{a2}, $${\smallmatrix 1\\2
\endsmallmatrix}{\smallmatrix 3
\endsmallmatrix}{\smallmatrix 2\\3~4
\endsmallmatrix}{\smallmatrix 4
\endsmallmatrix},~\quad{\smallmatrix 1\\2
\endsmallmatrix}{\smallmatrix 2\\4
\endsmallmatrix}{\smallmatrix 2\\3~4
\endsmallmatrix}{\smallmatrix 4
\endsmallmatrix},~\quad {\smallmatrix 1\\2
\endsmallmatrix}{\smallmatrix 3
\endsmallmatrix}{\smallmatrix 2\\3~4
\endsmallmatrix}{\smallmatrix 2\\3
\endsmallmatrix},~\quad {\smallmatrix 1\\2
\endsmallmatrix}{\smallmatrix 2\\4
\endsmallmatrix}{\smallmatrix 2\\3~4
\endsmallmatrix}{\smallmatrix 2
\endsmallmatrix},~\quad {\smallmatrix 1\\2
\endsmallmatrix}{\smallmatrix 2
\endsmallmatrix}{\smallmatrix 2\\3~4
\endsmallmatrix}{\smallmatrix 2\\3
\endsmallmatrix}$$ are all tilting $B$-modules
}\end{example}

Next, we will consider the relationship between $\stilt B$ and $\stilt A$. We need the following notions.

 Let $A$ be an algebra.
The {\it support $\tau$-tilting quiver}(or Hasse quiver) $\H(A)$ of $A$ is defined as follows (more  detail can be found  \cite[Definition 2.29]{AIR})

$\bullet$ vertices :  the isomorphisms classes of basic support $\tau$-tilting $A$-modules.

$\bullet$ arrows:  from a module to its left mutation.

It is well known that $\H(A)$ is a poset. Let $\mathcal{N}$ be a subposet of $\H(A)$ and $\mathcal{N}':=\H(A)\setminus \mathcal{N}$.  We define a new quiver $\H(A)^{\mathcal{N}}$  from $\H(A)$ as follows.

$\bullet$ vertices :   vertices in $\H(A)$ and $\mathcal{N}^+$ where $\mathcal{N}^+$ is a copy of $\mathcal{N}$.

$\bullet$ arrows:  $\{a_1\to a_2\mid a_1\to a_2 \in \mathcal{N}'\} \coprod \{n_2\to a_2\mid n_2\to a_2,  n_2\in \mathcal{N}, a_2\in \mathcal{N}'\}$

~~~~~~~~\qquad$\coprod \{n_1\to n_2, n^+_1\to n^+_2\mid n_1\to n_2\in \mathcal{N}\}$$\coprod \{a_1\to n^+_1\mid a_1\to n_1, n_1\in \mathcal{N}, a_1\in \mathcal{N}'\}$

~~~~~~~~\qquad$\coprod \{n^+_1\to n_1\mid n_1\in \mathcal{N}\}.$

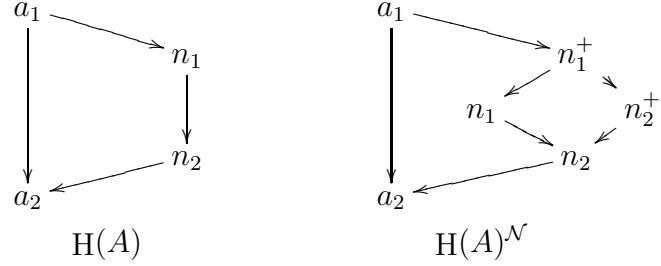
$$\xymatrix@R=2PT@C=5PT{a_1\ar[dddd]\ar[rrd]&&&&&&&&a_1\ar[dddd]\ar[rrd]&&&\\
&&n_1\ar[dd]&&&&&&&&n_1^+\ar[rd]\ar[ld]&\\
&&&&&&&&&n_1\ar[rd]&&n_2^+\ar[ld]\\
&&n_2\ar[lld]&&&&&&&&n_2\ar[lld]&\\
a_2&&&&&&&&a_2&&&\\
&\H(A)&&&&&&&&\H(A)^{\mathcal{N}}\\}.$$

Suppose that $A$ is  an algebra with an indecomposable projective-injective module $Q$.  Let $\overline{A}:=A/\Soc(Q)$ and
$$\mathcal{N}:=\{N\in \stilt \overline{A}\mid Q/\Soc(Q)\in \add N~\text{and} ~\Hom_A(N,Q)=0 \}.$$

The following Lemma can be found in \cite[Theorem 3.3]{Adachi2016}.

\begin{lemma}\label{2.5} Let $A$ be an algebra  with an indecomposable projective-injective module $Q$. Then there is an isomorphism of posets
$$\H(A)\longleftrightarrow\H{(\overline{A})}^{\mathcal{N}}.$$
\end{lemma}

Applying this result to the algebra $B$, we have the following

\begin{proposition}\label{2.6} Let $\mathcal{N}:=\{S_a\oplus L\mid L\in \stilt A/\langle e_i\rangle\}.$ Then there is an isomorphism of posets
$$\H(B)\longleftrightarrow \H{(A\times k)}^{\mathcal{N}}.$$

\end{proposition}

\begin{proof} Take  $Q=P_a$ which is an indecomposable projective-injective $B$-module. Since $\Soc P_a\cong S_i$, we have $\overline{B}=B/S_i\cong A\times k$ and $P_a/S_i\cong S_a$. We only need to show $\mathcal{N}=\{S_a\oplus L\mid L\in \stilt A/\langle e_i\rangle\}$  in Lemma \ref{2.5}. Note that

\begin{equation*}
\begin{split}
\mathcal{N}&=\{N\in \stilt \overline{B}\mid Q/\Soc(Q)\in \add N~\text{and} ~\Hom_B(N,Q)=0 \}\\
&=\{N\in \stilt (A\times k)\mid S_a\in \add N~\text{and} ~\Hom_B(N,P_a)=0 \}\\
&=\{S_a\oplus L\mid L\in \silt A~\text{and} ~\Hom_B(L,P_a)=0 \}\\
&=\{S_a\oplus L\mid L\in \silt A~\text{and} ~\Hom_A(L,S_i)=0 \}.
\end{split}
\end{equation*}
Since $i$ is a source point of $A$,  this implies $L$ has no $S_i$  as a composition factor and hence it  is exactly a support $\tau$-tilting $A/\langle e_i\rangle$-module. Thus,
$\mathcal{N}=\{S_a\oplus L\mid L\in \stilt A/\langle e_i\rangle\}.$

\end{proof}
\begin{corollary}\label{2.7} We have
$$ |\stilt B|=2|\stilt A|+|\stilt A/\langle e_i\rangle|.$$
\end{corollary}
\begin{proof}
According to the definition of  $\H{(A\times k)}^{\mathcal{N}}$, we have
$$|\H{(A\times k)}^{\mathcal{N}}|= |\H{(A\times k)}|+|{\mathcal{N}}|=2|\H(A)|+|{\mathcal{N}}|=2|\stilt A|+|\stilt A/\langle e_i\rangle|.$$ Therefore, $|\stilt B|=|\H(B)|=2|\stilt A|+|\stilt A/\langle e_i\rangle|$ by Proposition \ref{2.6}.
\end{proof}

Now, it is easy to draw the quiver of $\H(B)$ from the quiver of $\H(A)$ as follows.
$$\H(A)\to\H(A\times k)\to \H{(A\times k)}^{\mathcal{N}}\cong \H(B).$$
\begin{example} {\rm
Let $A$ be a finite dimensional $k$-algebra given by the quiver
$$2 {\longrightarrow} 1.$$ Considering the one-point extension of $A$ by the simple module corresponding to the point $2$,
the algebra $B=A[2]$ is given by the quiver
$$3 \stackrel{\alpha}{\longrightarrow} 2 \stackrel{\beta}{\longrightarrow} 1$$ with the relation  $\alpha\beta=0$.
We can get  the Hasse quiver $\H(B)$ of $B$ as follows where  $\mathcal{N}$ is remarked by red and
$\mathcal{N}^+$ by blue.
$$
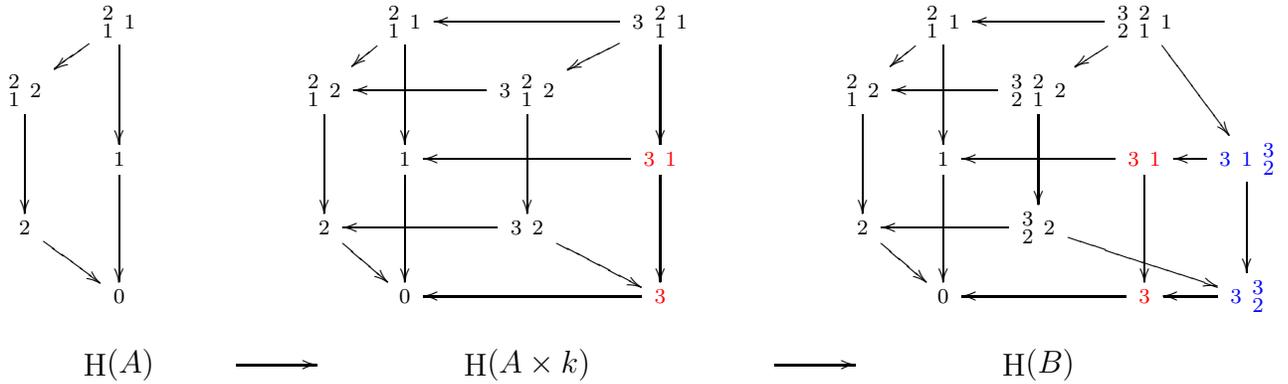
\xymatrix@R=8PT@C=8PT{&{\smallmatrix 2\\1
\endsmallmatrix}{\smallmatrix 1
\endsmallmatrix}\ar[dd]\ar[ld]&&&&&{\smallmatrix 2\\1
\endsmallmatrix}{\smallmatrix 1
\endsmallmatrix}\ar[dd]\ar[ld]&&{\smallmatrix 3
\endsmallmatrix}{\smallmatrix 2\\1
\endsmallmatrix}{\smallmatrix 1
\endsmallmatrix}\ar[dd]\ar[ld]\ar[ll]&&&&&{\smallmatrix 2\\1
\endsmallmatrix}{\smallmatrix 1
\endsmallmatrix}\ar[dd]\ar[ld]&&{\smallmatrix 3\\2
\endsmallmatrix}{\smallmatrix 2\\1
\endsmallmatrix}{\smallmatrix 1
\endsmallmatrix}\ar[ddr]\ar[ld]\ar[ll]&\\
{\smallmatrix 2\\1
\endsmallmatrix}{\smallmatrix 2
\endsmallmatrix}\ar[dd]&&&&&{\smallmatrix 2\\1
\endsmallmatrix}{\smallmatrix 2
\endsmallmatrix}\ar[dd]&&{\smallmatrix 3
\endsmallmatrix}{\smallmatrix 2\\1
\endsmallmatrix}{\smallmatrix 2
\endsmallmatrix}\ar[dd]\ar[ll]&&&&&{\smallmatrix 2\\1
\endsmallmatrix}{\smallmatrix 2
\endsmallmatrix}\ar[dd]&&{\smallmatrix 3\\2
\endsmallmatrix}{\smallmatrix 2\\1
\endsmallmatrix}{\smallmatrix 2
\endsmallmatrix}\ar[dd]\ar[ll]&&\\
&{\smallmatrix 1
\endsmallmatrix}\ar[dd]&&\ar@{}[r]&&&{\smallmatrix 1
\endsmallmatrix}\ar[dd]&&{\smallmatrix \color{red}{3}
\endsmallmatrix}{\smallmatrix \color{red}{1}
\endsmallmatrix}\ar[dd]\ar[ll]&&&&&{\smallmatrix 1
\endsmallmatrix}\ar[dd]&&{\smallmatrix \color{red}{3}
\endsmallmatrix}{\smallmatrix \color{red}{1}
\endsmallmatrix}\ar[dd]\ar[ll]&{\smallmatrix \color{blue}{3}
\endsmallmatrix}{\smallmatrix \color{blue}{1}
\endsmallmatrix}{\smallmatrix \color{blue}{3}\\\color{blue}{2}
\endsmallmatrix}\ar[dd]\ar[l]\\
{\smallmatrix 2
\endsmallmatrix}\ar[dr]&&&&&{\smallmatrix 2
\endsmallmatrix}\ar[dr]&&{\smallmatrix 3
\endsmallmatrix}{\smallmatrix 2
\endsmallmatrix}\ar[dr]\ar[ll]&&&&&{\smallmatrix 2
\endsmallmatrix}\ar[dr]&&{\smallmatrix 3\\2
\endsmallmatrix}{\smallmatrix 2
\endsmallmatrix}\ar[drr]\ar[ll]&&\\
&{\smallmatrix0
\endsmallmatrix}&&&&&{\smallmatrix0
\endsmallmatrix}&&{\smallmatrix \color{red}{3}
\endsmallmatrix}\ar[ll]&&&&&{\smallmatrix0
\endsmallmatrix}&&{\smallmatrix \color{red}{3}
\endsmallmatrix}\ar[ll]&{\smallmatrix \color{blue}{3}
\endsmallmatrix}{\smallmatrix \color{blue}{3}\\  \color{blue}{2}
\endsmallmatrix}\ar[l]\\
&\H(A)&&\ar[rr]&&&&\H(A\times k)&&&\ar[rr]&&&&\H(B)&&
}$$}
\end{example}

The linearly Dynkin type algebras be the following quivers.
 $$\xymatrix@!@R=5pt@C=5pt{
A_n:&n\ar[r]&n-1\ar[r]&\cdots\ar[r]&2\ar[r]&1}$$

$$\xymatrix@R=5PT@C=15PT{&&&&1\\
D_n:~~~n\ar[r]&n-1\ar[r]&\cdots\ar[r]&3\ar[ru]\ar[rd]&\\
&&&&2}$$

Take $A_n^2:=kA_n/\rad^2$  and $D_n^2:=kD_n/\rad^2$.  Applying our results, we can give   recurrence relations about the numbers of $\tau$-tilting modules and support $\tau$-tilting modules over $A_n^2$ and $D_n^2$.

\begin{theorem} Let $\Lambda^2_n$  be an algebra ($A_n^2$ or $D_n^2$). Then we have
\begin{enumerate}
\item[(1)]  $|\tilt \Lambda^2_n|=|\tilt \Lambda^2_{n-1}|+|\tilt \Lambda^2_{n-2}|.$
\item[(2)] $|\stilt \Lambda^2_n|=2|\stilt \Lambda^2_{n-1}|+|\stilt \Lambda^2_{n-2}|.$
\end{enumerate}
\end{theorem}
\begin{proof}
Since $\Lambda^2_n$ is the one-point extension of $\Lambda^2_{n-1}$ by simple module $S_{n-1}$ and  $\Lambda^2_{n-1}/\langle e_{n-1}\rangle\cong \Lambda^2_{n-2}$. Now, the result follows from Corollary \ref{2.4} and Corollary \ref{2.7}.
\end{proof}

\begin{corollary}~
\begin{enumerate}
\item[(1)]  $|\t A^2_n|=2~~(n\geqslant 2)$.
\item[(2)]  $|\tilt A^2_n|=\frac{(1+\sqrt{5})^{n+1}-(1-\sqrt{5})^{n+1}}{\sqrt{5}\cdot 2^{n+1}}.$
\item[(3)] $|\stilt A^2_n|=\frac{(1+\sqrt{2})^n-(1-\sqrt{2})^n}{2\sqrt{2}}.$
\item[(4)] $|\t D^2_n|=5.$
\item[(5)]  $|\tilt D^2_n|=\frac{(2\sqrt{5}-1)(1+\sqrt{5})^{n-1}+(2\sqrt{5}+1)(1-\sqrt{5})^{n-1}}{\sqrt{5}\cdot 2^{n-1}}.$
\item[(6)] $|\stilt D^2_n|=\frac{(3\sqrt{2}-1)(1+\sqrt{2})^{n-1}+(3\sqrt{2}+1)(1-\sqrt{2})^{n-1}}{\sqrt{2}}.$
\end{enumerate}
\end{corollary}

\begin{example} We give some examples of the numbers of $\tau$-tilting modules and support $\tau$-tilting modules over $A_n^2$ and $D_n^2$ in the following tables.
\begin{table}[h]
\begin{center}
\begin{tabular}{l|cccccccccccc}
 \hline
$n$& 1&2&3&4&5&6&7&8&9&10\\
  \hline
$|\tilt A^2_n|$ & 1&2&3&5&8&13&21&34&55&89\\
$|\stilt A^2_n|$ &2&5 & 12 &29&70&169&408&985&2378&5741\\
 \hline \end{tabular}
 \end{center}
\end{table}

\begin{table}[h]
\begin{center}
\begin{tabular}{l|cccccccccccc}
 \hline
$n$& 4&5&6&7&8&9&10\\
  \hline
$|\tilt D^2_n|$ & 6&11&17&28&45&73&118\\
$|\stilt D^2_n|$ &32& 78 & 118 &454&1026&2506&6038\\
 \hline \end{tabular}
 \end{center}
\end{table}

\end{example}

\section*{Acknowledgements}
The author would like to thank Professor Zhaoyong Huang for helpful discussions. This work was partially supported by the National natural Science Foundation of China (No. 11971225).

\end{document}